\documentclass[12pt,reqno]{amsproc}
\usepackage{times}

\usepackage{amsmath,amsthm,amssymb,wasysym,mathtools}
\usepackage[inline]{enumitem}
\usepackage{caption}
\usepackage{subcaption}
\usepackage{graphicx}

\DeclareMathOperator{\RE}{Re}

\numberwithin{equation}{section}

\usepackage{amsthm}

\newtheoremstyle{fancytheorem}               % name of theoremstyle
{.5\baselineskip±.2\baselineskip}           % Space above, these are standard values for AMS class document
{.5\baselineskip±.2\baselineskip}           % Space below
{\itshape\addtolength{\leftskip}{0mm}\setlength{\parindent}{0em}}      % Body font +indent added +no indent between paragraphs
{0mm}                                     % Indent of header amount
{\bfseries}                                 % Theorem head font
{}                                  % Punctuation after theorem head and start at new line
{.5em}                                      % Space after theorem head
{\thmname{#1}\thmnumber{ #2}. \thmnote{[#3]}}% Theorem head spec (can be left empty, meaning ‘normal’)

\newtheoremstyle{fancydefinition}               % name of theoremstyle
{.5\baselineskip±.2\baselineskip}           % Space above, these are standard values for AMS class document
{.5\baselineskip±.2\baselineskip}           % Space below
{\upshape\addtolength{\leftskip}{0mm}\setlength{\parindent}{0em}}      % Body font +indent added +no indent between paragraphs
{0mm}                                     % Indent of header amount
{\bfseries}                                 % Theorem head font
{}                                  % Punctuation after theorem head and start at new line
{.5em}                                      % Space after theorem head
{\thmname{#1}\thmnumber{ #2}. \thmnote{[#3]}}% Theorem head spec (can be left empty, meaning ‘normal’)

\theoremstyle{fancytheorem}

\newtheorem{theorem}{Theorem}[section]

\newtheorem{remark}[theorem]{Remark}

\theoremstyle{fancydefinition}
\newtheorem{definition}[theorem]{Definition}

\allowdisplaybreaks

\begin{document}\vspace{-4cm}
	\title[On functions starlike with respect to $n$-ply points]{On functions starlike with respect to $n$-ply symmetric, conjugate and symmetric conjugate points}
	
	\author[S. Malik]{Somya Malik}
	\address{Department of Mathematics \\National Institute of Technology\\Tiruchirappalli-620015,  India }
	%\email{arya.somya@gmail.com}
	
	\author[V. Ravichandran]{Vaithiyanathan Ravichandran}
	%\address{Department of Mathematics \\National Institute of Technology\\Tiruchirappalli-620015,  India }
	\email{arya.somya@gmail.com; vravi68@gmail.com, ravic@nitt.edu}
	
	\begin{abstract}
		For given non-negative real numbers $\alpha_k$ with $ \sum_{k=1}^{m}\alpha_k =1$ and normalized analytic functions $f_k$, $k=1,\dotsc,m$, defined on the open unit disc,  let the functions $F$  and $F_n$ be defined by $ F(z):=\sum_{k=1}^{m}\alpha_k f_k (z)$, and $F_{n}(z):=n^{-1}\sum_{j=0}^{n-1} e^{-2j\pi  i/n} F(e^{2j\pi  i/n} z)$. This paper studies the functions $f_k$ satisfying the subordination $zf'_{k} (z)/F_{n} (z)  \prec h(z)$ where  the function $h$  is a convex univalent function  with positive real part. We also consider the analogues of the classes of starlike   functions with respect to symmetric, conjugate, and symmetric conjugate points.  Inclusion and convolution results are proved for these and related classes. Our classes generalize several well-known classes and connection with the previous works are indicated.
	\end{abstract}
	
	

	\thanks{The first author is supported by  the UGC-JRF Scholarship.}
	
	\dedicatory{Dedicated to Prof.  Bal Kishan Dass on the occasion of his 70th birthday}
	\maketitle
	
	\section{Introduction}
	Let $\mathcal{A}$ be the class of normalized analytic functions $f$ of the form $f(z)=z+\sum_{k=2}^{\infty}a_{k}z^{k}$ defined on the unit disc $\mathbb{D}=\left\{z\in \mathbb{C} : |z|<1\right\}$. For analytic functions $f$ and $g$,\  $f\prec g$ (read: $f$ is subordinate to $g$) if  the condition $f(z)=g(w(z))$ holds for some  analytic function $w:\mathbb{D}\to \mathbb{D}$  with $w(0)=0 $. If the function $g$ is univalent, then $f \prec g $ if and only if $f(0)=g(0)$ and $f(\mathbb{D})\subset g(\mathbb{D})$. The usual subclasses of starlike and convex functions are   characterized respectively by the inequalities $\RE({zf'(z)}/{f(z)})>0$ and $ 1+\RE({zf''(z)}/{f'(z)})>0$. These classes as well as several subclasses of starlike functions and convex functions  are respectively characterized by the subordination \[\dfrac{zf'(z)}{f(z)}\prec h(z)\quad \text{or}\quad 1+\dfrac{zf''(z)}{f'(z)}\prec h(z),\]
	where the function $h$ is a univalent function with positive real part. The classes of functions $f\in \mathcal{A}$ satisfying these subordinations are respectively denoted by $\mathcal{ST}(h)$ and $\mathcal{CV}(h)$. For $0\leqslant \alpha<1$, let $h_\alpha:\mathbb{D}\to\mathbb{C}$ be defined by $h_\alpha(z)=(1+(1-2\alpha )z)/(1-z)$;  the functions in  $\mathcal{ST}(\alpha):=\mathcal{ST}(h_\alpha) $ and $\mathcal{CV}(\alpha):=\mathcal{CV}( h_\alpha)  $ are respectively the usual starlike and convex functions of order $\alpha$.   A function  $f\in\mathcal{A}$ is  starlike with respect to symmetric points if for every $r$ less
	than and sufficiently close to 1 and every $\zeta$
	on $|z|=r$, the angular velocity
	of $f(z)$ about the point $f(-\zeta)$ is positive at $ z=\zeta$ as $z$ traverses the circle
	$|z|=r$ in the positive direction. Analytically, a function   $f\in\mathcal{A}$ is  starlike with respect to symmetric points if
	\[ \RE \frac{zf'(z)}{f(z)-f(-\zeta)}>0\quad \text{for}\quad z=\zeta,\  |\zeta|=r\] or
	equivalently if it satisfies the inequality
	\begin{equation}\label{eqn1.1}
		\RE \dfrac{zf'(z)}{f(z)-f(-z)} >0\quad \text{for all } z\in \mathbb{D};
	\end{equation}
	these functions were  introduced and studied by Sakaguchi \cite{saka}. He also gave a few generalizations of \eqref{eqn1.1} using the $n$-ply points function $f_n(z)=\sum_{j=0}^{n-1} \varepsilon ^{-j} f(\varepsilon^{j}z)$ for $f\in \mathcal{A}$, where $\varepsilon=\mathit{e}^{2\pi i /n}$, $n$ is a positive integer. Properties of integral operators, structural formulas, distortion estimates and subordination results for functions defined by  using  $f_n$  are investigated  in \cite{Mocanu,NPon,PadThang}. The Feketo-Szego inequalities and their applications to classes defined through convolution is investigated  by Shanmugam  et al.\ \cite{ShanRamaRavi}. Recently, Gochhayat and  Prajapati \cite{Goch} studied growth, distortion and radius problem for general class of symmetric points. The classes of starlike functions with respect to conjugate points, and starlike functions with respect to symmetric conjugate points were due to El-Ashwah and Thomas \cite{ElashDK}; these  classes have functions $f$ satisfying respectively the inequalities \[\RE \dfrac{zf'(z)}{f(z)+\overline{f(\overline{z})}}  >0, \quad \text{and}\quad \RE \dfrac{zf'(z)}{f(z)-\overline{f(-\overline{z})}}  >0.\]
	Kasi \cite{kasi} solved some radius problems related to these classes.

	Convolution or the Hadamard product of two functions $f$ and $g$ in $\mathcal{A}$, given by $f(z)=z+\sum_{k=2}^{\infty}a_{k}z^{k}$ and $g(z)=z+\sum_{k=2}^{\infty}b_{k}z^{k}$, is the function $f*g$ given by $(f\ast g)(z)=z+\sum_{k=2}^{\infty}a_{k}b_{k}z^{k}$. Barnard and Kellogg \cite{BerKellogg} gave applications of the theory of convolution for finding the radii of certain well known classes; these ideas and techniques have been widely used. Al-Amiri \cite{amiri} found the radii of the classes he defined, under certain very well known operators. One way to unify these classes of starlike functions with respect to symmetric, conjugate and symmetric conjugate points is the observation that these functions  satisfy
	\[ \RE \frac{zf'(z)}{f(z)+g(z)}>0 \quad \text{and}\quad  \RE \frac{zg'(z)}{f(z)+g(z)}>0\]
	where $g(z)=-f(-z)$, $\overline{f(\overline{z})}$ and $-\overline{f(-\overline{z})}$ respectively. More generally, for given non-negative real numbers $\alpha_k$ with $ \sum_{k=1}^{m}\alpha_k =1$ and normalized analytic functions $f_k:\mathbb{D}\to\mathbb{C}$, $k=1,\dotsc,m$,  let the function $F$ be defined by $ F(z):=\sum_{k=1}^{m}\alpha_k f_k (z)$. Then the general form of the condition is
	$\RE(zf'_k(z)/F(z))>0$.  Using the function $F_n$ obtained from $F$, the subordination relation becomes $zf'_{k} (z)/F_{n} (z)  \prec h(z)$ where  the function $h$  is a convex univalent function  with positive real part.	In this paper, we study these classes as well as other classes related to  symmetric, conjugate and symmetric conjugate  points. We prove inclusion and convolution theorems for these classes. The  results obtained include several known results.
	
	The main tool used in the proofs is a convolution theorem \cite[Theorem 2.4]{rusch1} for pre-starlike functions due to Ruscheweyh. 	For $\alpha <1$, the generalized Koebe function $k_\alpha:\mathbb{D}\to\mathbb{C}$ is defined by
	\[ k_\alpha(z):= \frac{z}{(1-z)^{2-2\alpha}} . \] A function $f\in \mathcal{A}$ is called a pre-starlike function of order $\alpha$ if its convolution $f*k_\alpha$ with the generalized Koebe function $k_\alpha$  is starlike of order $\alpha$.   A pre-starlike function of order 1 is a function $f\in \mathcal{A}$ satisfying the inequality $\RE (f(z)/z)>1/2$.
	The class of pre-starlike functions of order $\alpha$ is denoted by $\mathcal{R}_\alpha$. The next theorem has vast applications; for example, it follows  that  the classes of starlike functions of order $\alpha$ and convex functions of order $\alpha$ are closed  under convolution with functions in the class $\mathcal{R}_\alpha$.
	
	\begin{theorem} \cite[Theorem 2.4]{rusch1}\label{thm1.1}
		For $\alpha\leqslant 1$, let the function $\phi \in \mathcal{R}_\alpha$, and the function $f\in \mathcal{ST}(\alpha)$.  If  the function $H$ is analytic in $\mathbb{D} $,  then the  image of the unit disc $\mathbb{D}$ under the function $(\phi \ast (H f))/(\phi \ast f)$ is contained in the closed convex hull $\overline{co}(H(\mathbb{D})) $ of $H(\mathbb{D})$:
		\begin{equation} \label{Ruschth}
			\dfrac{\phi \ast (H f)}{\phi \ast f}(\mathbb{D})\subset \overline{co}(H(\mathbb{D})).\end{equation}
	\end{theorem}
	
	If the function $H$ is convex, then $\overline{co}(H(\mathbb{D}))=\overline{ H(\mathbb{D})}$ and so  the containment in \eqref{Ruschth} becomes the subordination  $ (\phi \ast (H f))/(\phi \ast f) \prec H$. We also need the following theorem of Miller and Mocanu \cite[Corollary 4.1h.1]{MillMoca}.
	
	\begin{theorem}\label{thm1.2}
		Let the function $h$ be   convex in $\mathbb{D}$,    the functions $S$ and $T$ be analytic in $\mathbb{D}$ and $S(0)=T(0)$. If the inequality $\RE (zS'(z)/S(z))>0$ holds and
		${T'(z)}/{S'(z)}\prec h(z)$, then ${T(z)}/{S(z)} \prec h(z)$.
	\end{theorem}
	
	\section{The Classes $\mathcal{ST}_{n,m} (h)$ and $\mathcal{CV}_{n,m} (h)$}
	
	For a given positive integer $m$, let the class $\mathcal{A}^{m}$ be defined by $\mathcal{A}^{m}:=\{\widehat{f}:=(f_1,f_2,\dotsc,f_m): f_{k}\in \mathcal{A},\ 1 \leqslant k \leqslant m\}$. For $0\leqslant \alpha_k \leqslant 1$, with $\sum_{k=1}^{m}\alpha_k =1$, and the function $\widehat{f}\in \mathcal{A}^{m}$, define the function $F:\mathbb{D}\to\mathbb{C}$ given by
	\begin{equation}\label{eqn2.1}
		F(z)=\sum_{k=1}^{m}\alpha_k f_k (z).
	\end{equation}
	For an integer $n\geqslant 1$, the $n$-ply points function $F_n: \mathbb{D}\to\mathbb{C}$ corresponding to the function $F$ is defined by
	\begin{equation}\label{eqn2.2}
		F_{n}(z):=\frac{1}{n}\sum_{j=0}^{n-1}\varepsilon^{-j} F(\varepsilon^{j} z)
	\end{equation}
	where $\varepsilon  =\mathit{e}^{2\pi i /n}$ is the primitive $n$th root of unity.
	For $\widehat{f}\in \mathcal{A}^{m}$ and $g\in\mathcal{A}$, the convolution
	$\widehat{f} *g$ is defined term-wise   by  $\widehat{f}\ast g:=( f_1 \ast g, f_2 \ast g, \ldots , f_m \ast g )$ and the ordinary product
	$g\widehat{f}:=( gf_1 ,g f_2 , \ldots ,g f_m  )$. The derivative $\widehat{f}'$ is defined by $\widehat{f}':=(f_1',f_2',\dotsc,f_m')$ and the conjugate $\overline{\widehat{f}}$ by $\overline{\widehat{f}}:=(\overline{f_1},\overline{f_2},\dotsc,\overline{f_m})$.  With these notations, the following classes of starlike and convex functions related to $\widehat{f}$ and $F_{n}$ are defined.
	
	\begin{definition}
		Let $h$ be a convex univalent function with positive real part. The classes of starlike and convex functions related to $\widehat{f}$ and $F_{n}$, denoted by $\mathcal{ST}_{n,m} (h)$ and $\mathcal{CV}_{n,m} (h)$ consist of functions $\widehat{f}\in \mathcal{A}^{m}$ satisfying respectively the subordinations \[  \dfrac{zf'_{k} (z)}{F_{n} (z)} \prec h(z)\quad \text{and}\quad 	\dfrac{(zf'_{k})'(z)}{F'_{n}(z)} \prec h(z).\]
		For a fixed $g\in \mathcal{A}$,  the class $\mathcal{ST}_{n,m} (g,h)$ consists of all  $ \widehat{f}$ for which $ \hat{f}\ast g \in \mathcal{ST}_{n,m} (h)$   and     the class $\mathcal{CV}_{n,m} (g,h)$ consists of all functions $\widehat{f}$ for which $\widehat{f}\ast g\in \mathcal{CV}_{n,m} (h)$.
	\end{definition}
	
	For $n=1$, and $\alpha_k=1/m$ for all $1\leqslant k\leqslant m$, the classes $\mathcal{ST}_{n,m}(g,h)$ and $\mathcal{CV}_{n,m}(g,h)$ reduce to the classes $\mathcal{S}^{*}_{m}(g,h)$ and $\mathcal{K}_{m}(g,h)$ respectively; investigated in \cite{RosMahRaviSub}. For $g(z)=z/(1-z)^{a}$, with all the afore mentioned particular values for $n$ and $\alpha_k$, the class $\mathcal{ST}_{n,m}(g,h)$ becomes the class $\mathcal{S}_{a}(h)$ studied in \cite{PadPar}. The class $\mathcal{CV}_{n,m}(g,h)$ is the class $\mathcal{K}_{a}(h)$ studied in \cite{PadMan}. For $n=m=1$, the two classes $\mathcal{ST}_{n,m}(g,h)$ and $\mathcal{CV}_{n,m}(g,h)$ are reduced respectively to  $\mathcal{S}^{*}(g,h)$ and $\mathcal{K}(g,h)$ studied in \cite{shan}.
	
	It follows immediately that  $\mathcal{ST}_{n,m} (z/(1-z),h)=   \mathcal{ST}_{n,m} (h)$, $\mathcal{CV}_{n,m} (z/(1-z),h)=   \mathcal{CV}_{n,m} (h)$ and $\mathcal{ST}_{n,m} (z/(1-z)^2,h)=   \mathcal{CV}_{n,m} (h)$.  Also, the following Alexander relationship holds: $\widehat{f}\in \mathcal{CV}_{n,m} (h)$ if and only if  $z\widehat{f'}\in \mathcal{ST}_{n,m}(h)$ and  $\widehat{f}\in \mathcal{CV}_{n,m} (g,h)$ if and only if  $z\widehat{f'}\in \mathcal{ST}_{n,m}(g,h)$.
	
	The first theorem in this section discusses the starlikeness of the $n$-ply points function $F_{n}$.
	
	\begin{theorem}\label{thm2.2} If  the function   $h$ is    a convex univalent function with positive real part satisfying $h(0)=1$ then,  for $ 1\leqslant k\leqslant m$, the component  $f_k$ of $\widehat{f} \in  \mathcal{ST}_{n,m}(h)$  is close-to-convex and hence univalent. Also, the $n$-ply points function $F_n $ belongs to the class $\mathcal{ST}(h)$.
	\end{theorem} 	
	
	\begin{proof}For a given  $\widehat{f}\in \mathcal{ST}_{n,m} (h)$, it is first shown that the $n$-ply points function $F_n $ given by \eqref{eqn2.2}  is starlike. The definition of the class $\mathcal{ST}_{n,m}(h)$ gives $zf'_{k} (z)/F_{n} (z) \in  h(\mathbb{D})$ for each $z\in \mathbb{D}$. Since $ \alpha_k$'s are non-negative real numbers with $\sum_{k=1}^{m} \alpha_k=1$, the convexity of $h(\mathbb{D})$ readily shows that $\sum_{k=1}^{m} \alpha_k \left(zf'_{k} (z)/F_{n} (z)\right) \in h(\mathbb{D})$ for each $z\in \mathbb{D}$.  The definition of $F$ given in \eqref{eqn2.1} implies that  $zF' (z)/F_{n}(z) \in h(\mathbb{D})$ for each $z\in \mathbb{D}$.
		
		The Taylor series of the function $F_n$ in \eqref{eqn2.2}, gives that $F_{n}(\varepsilon ^{j} z)=\varepsilon ^{j} F_{n}(z)$. Using this identity and
		replacing $z$ by $\varepsilon ^{j} z$  in  $zF' (z)/F_{n}(z) \in h(\mathbb{D})$,  it  is easy to infer that,
		\begin{equation}\label{eqn2.3}
			\dfrac{zF'(\varepsilon ^{j} z)}{F_{n}(z)} \in h(\mathbb{D})
		\end{equation}
		for each $z\in \mathbb{D}$. From \eqref{eqn2.2}, it follows that $F_{n}'(z):=\sum_{j=0}^{n-1}\ (F'(\varepsilon^{j} z))/n$. Using this and the convexity of $h(\mathbb{D})$,  \eqref{eqn2.3} yields  \[\frac{zF_{n}'(z)}{F_{n}(z)} =
		\frac{1}{n}\sum_{j=0}^{n-1} \dfrac{zF'(\varepsilon ^j z)}{F_{n}(z)} \in h(\mathbb{D})\]
		for each $z\in\mathbb{D}$.
		Thus,  $ zF_{n}'(z) /F_{n}(z)  \prec h(z)$ or $F_{n}\in \mathcal{ST}(h)$. Since $h$ is a function with positive real part, it follows that $\mathcal{ST}(h)\subset \mathcal{ST}$ and thus $F_{n} \in \mathcal{ST}$.
		
		Since $h$ is a function with positive real part, it follows that the function $zf'_{k} (z)/F_{n} (z)$ has positive real part in $\mathbb{D}$. The starlikeness of $F_n$ immediately shows that the function $f_k\in \mathcal{A}$ is close-to-convex and univalent.
	\end{proof}
	
	Containment relation between the classes $\mathcal{CV}_{n,m}(h)$ and $\mathcal{ST}_{n,m}(h)$ is discussed in the following theorem.
	\begin{theorem}\label{thm2.3} If   $h$ is  a convex univalent function with positive real part satisfying $h(0)=1$, then   $\mathcal{CV}_{n,m}(h) \subset \mathcal{ST}_{n,m}(h)$ and,  for $ 1\leqslant k\leqslant m$, the component  $f_k$ of $\widehat{f} \in  \mathcal{CV}_{n,m}(h)$  is close-to-convex and hence univalent.  Also, the $n$-ply points function $F_n $ belongs to the class $\mathcal{CV}(h)$.
	\end{theorem} 	
	
	\begin{proof}
		For $\widehat{f} \in \mathcal{CV}_{n,m} (h)$, it follows directly from the definition that  $(zf'_{k})'(z)/F'_{n}(z)  \prec  h(z)$. By following arguments similar to the ones used in the proof of Theorem~\ref{thm2.2}, it can be shown that $F_n\in \mathcal{CV}(h)$. Since $h$ has positive real part, it follows that  $F_n$ is convex  and in particular  starlike. Applying Theorem \ref{thm1.2} yields $ zf'_{k}(z)/F_n(z) \prec h(z)$, proving  $\widehat{f} \in \mathcal{ST}_{n,m}(h)$. Close-to-convexity and univalence of $f_k$ follows from Theorem~\ref{thm2.2}.
	\end{proof}
	
	A class $\mathcal{F}$ is said to be closed under convolution with functions in the class $\mathcal{G}$ if $f*g\in \mathcal{F}$ for each $f\in \mathcal{F}$ and $g\in \mathcal{G}$. If $\mathcal{G}=\mathcal{CV}$ or $\mathcal{R}_\alpha$, then the statement becomes respectively that $\mathcal{F}$ is closed under convolution with convex functions or closed under convolution with pre-starlike functions of order $\alpha$.
	Ruschweyeh and Sheil-Small \cite{ruscSS} have shown  that the classes of convex functions, starlike functions and close-to-convex functions are closed under convolution with convex functions. The next theorem extends this result to  the classes $\mathcal{ST}_{n.m}(g,h)$ and $\mathcal{CV}_{n,m}(g,h)$; each of these classes is shown to be closed under convolution with pre-starlike functions of order $\alpha$ where $h$ is a convex univalent function with real part greater than $\alpha$.

	\begin{theorem}\label{thm2.4} Let $g$ be a fixed function in $\mathcal{A}$. Let $h$ be a convex univalent function satisfying $\RE(h(z))>\alpha,\ h(0)=1,\ 0\leqslant \alpha < 1$. Then the classes $\mathcal{ST}_{n,m} (g,h)$, $\mathcal{CV}_{n,m} (g,h)$ and, in particular, the classes $\mathcal{ST}_{n,m} (h)$ and $\mathcal{CV}_{n,m} (h)$ are closed under convolutions with pre-starlike functions of order $\alpha$.
	\end{theorem}
	
	\begin{proof}
		First it is proved that the class $\mathcal{ST}_{n,m} (h)$ is closed under convolutions with pre-starlike functions of order $\alpha$; in other words, $\widehat{f} \ast \phi \in \mathcal{ST}_{n,m} (h)$ for each $\widehat{f} \in \mathcal{ST}_{n,m} (h)$ and each $\phi \in \mathcal{R}_\alpha$.	For $k=1,2\ldots,m$ and $n\geqslant 1$, define the function $H_{k}:\mathbb{D}\to\mathbb{C}$ by
		\[  H_{k}(z):=\dfrac{zf'_{k}(z)}{F_{n}(z)}.\]
		By Theorem \ref{thm2.2}, the function $F_n\in \mathcal{ST}(h)$. Since $\RE (h(z))>\alpha$, it follows that  $F_n\in \mathcal{ST} (\alpha)$. 		An application of Theorem \ref{thm1.1} shows that
		\begin{equation}\label{eqn2.4a} \dfrac{(\phi \ast (H_{k} F_n))(z)}{(\phi \ast F_n)(z)} \subset \overline{co}(H_{k}(\mathbb{D})). \end{equation}
		Since
		\begin{equation}\label{eqn2.4}
			\dfrac{(\phi \ast (H_{k} F_n))(z)}{(\phi \ast F_n)(z)}
			= \dfrac{(\phi \ast zf'_{k})(z)}{(\phi \ast F_{n})(z)}
			= \dfrac{z(\phi \ast f_{k})'(z)}{(\phi \ast F)_{n}(z)}
		\end{equation}
		for every $z\in \mathbb{D}$  and  $H_{k}(z)\prec h(z)$, it follows from \eqref{eqn2.4a} and  \eqref{eqn2.4}  that
		\[\dfrac{z(\phi \ast f_{k})'(z)}{(\phi \ast F)_{n}(z)} \prec h(z).\]
		This proves that $ \phi \ast \widehat{f} \in \mathcal{ST}_{n,m} (h)$.

		If $\widehat{f} \in \mathcal{ST}_{n,m} (g,h)$, then $\widehat{f} \ast g\in \mathcal{ST}_{n,m}(h)$. Since it is already proved that   the class $\mathcal{ST}_{n,m} (h)$ is closed under convolution with the class $\mathcal{R}_{\alpha}$, the associative property of convolution gives $ (\widehat{f} \ast \phi  )\ast g= (\widehat{f} \ast g )\ast \phi\in \mathcal{ST}_{n,m} (h) $ for $\phi\in \mathcal{R}_\alpha$. Therefore,  $\widehat{f} \ast \phi \in \mathcal{ST}_{n,m}(g,h)$.
		
		Finally, the statement $\widehat{f} \ast \phi \in \mathcal{CV}_{n,m} (g,h)$ for $\widehat{f} \in \mathcal{CV}_{n,m} (g,h)$ and $\phi\in\mathcal{R}_\alpha$ is to be proved. Let $\widehat{f} \in \mathcal{CV}_{n,m} (g,h)$. Then, by the Alexander relationship between the  classes $\mathcal{ST}_{n,m} (g,h)$ and $\mathcal{CV}_{n,m} (g,h)$, it is inferred that $  z\widehat{f'} \in \mathcal{ST}_{n,m}(g,h)$. Using what is already proved for $\mathcal{ST}_{n,m} (g,h)$, it follows that $(z \widehat{f'} \ast \phi) \in \mathcal{ST}_{n,m} (g,h)$. Since  $(z \widehat{f'} \ast \phi)=z(\phi \ast \widehat{f})'$,\ we have $z(\phi \ast \widehat{f})'\in \mathcal{ST}_{n,m} (g,h)$. Again, the Alexander relationship between the classes $\mathcal{ST}_{n,m} (g,h)$ and $\mathcal{CV}_{n,m} (g,h)$, implies that $\phi \ast \widehat{f} \in \mathcal{CV}_{n,m}(g,h)$. Since $\mathcal{CV}_{n,m}(z/(1-z),h)=\mathcal{CV}_{n,m}(h)$, the closure under convolution with pre-starlike functions of order $\alpha$ for the class $\mathcal{CV}_{n,m}(h)$ follows.
	\end{proof}
	
	A function $f\in \mathcal{A}$ is pre-starlike of order $0$ if $zf'=f\ast k_{0}= f\ast (z/(1-z)^2)$ is starlike, or $f$ is convex. Therefore, $\mathcal{R}_{0}=\mathcal{CV}$ is the class of convex functions. Similarly, $\mathcal{R}_{1/2}=\mathcal{ST}(1/2)$ is the class of starlike functions of order $1/2$. Assuming $h$ to be just a function with positive real part in Theorem~\ref{thm2.4}, implies that each of the classes $\mathcal{ST}_{n,m} (g,h),\ \mathcal{ST}_{n,m} (h),\ \mathcal{CV}_{n,m} (g,h),\ \mathcal{CV}_{n,m} (h)$  is closed under convolution with convex functions. The next section talks about the class of functions obtained with regards to the concept of symmetric points.
	
	\section{The Classes $\mathcal{ST S}_{n,m} (h)$ and $\mathcal{CV S}_{n,m} (h)$}
	Let $m,n$ be positive integers. By using the definition of the $n$-ply points function $F_n$ in \eqref{eqn2.2}  for $\widehat{f}\in \mathcal{A}^{m}$; the classes of starlike and convex function with respect to symmetric points related to $\widehat{f}$ and $F_{n}$,  $\mathcal{ST S}_{n,m} (h)$ and $\mathcal{CV S}_{n,m} (h)$ respectively are defined by
	\[\mathcal{ST S}_{n,m} (h)=\left\{\widehat{f} \in \mathcal{A}^{m}\ |\ \dfrac{2zf'_{k} (z)}{F_{n}(z)-F_{n}(-z)}\prec h(z)\right\},\]
	and
	\[\mathcal{CV S}_{n,m} (h)=\left\{\widehat{f} \in \mathcal{A}^{m}\ |\ \dfrac{(2zf'_{k})'(z)}{F'_{n}(z)+F'_{n}(-z)} \prec h(z)\right\},\]
	where $h$ is a convex univalent function with positive real part. For a fixed $g\in \mathcal{A}$, the class $\mathcal{ST S}_{n,m}(g,h)$ contains  all $\widehat{f}$ for which $\widehat{f} \ast g\in \mathcal{ST S}_{n,m}(h)$, and the class $\mathcal{CV S}_{n,m} (g,h)$ has all functions $\widehat{f}$ for which $\widehat{f}\ast g\in \mathcal{CV S}_{n,m} (h)$.
	
	From the above definitions, for $g(z)=z/(1-z)$,\ \  $\mathcal{ST S}_{n,m}(g,h)=\mathcal{ST S}_{n,m}(h)$, $\mathcal{CV S}_{n,m} (g,h)=\mathcal{CV S}_{n,m} (h)$ and $\mathcal{ST S}_{n,m} (z/(1-z)^2,h)=\mathcal{CV S}_{n,m} (h)$. Also, the Alexander relationship holds: $\widehat{f}\in \mathcal{CV S}_{n,m} (h)$ if and only if $z\widehat{f'}\in \mathcal{ST S}_{n,m}(h)$ and  $\widehat{f}\in \mathcal{CV S}_{n,m} (g,h)$ if and only if  $z\widehat{f'}\in \mathcal{ST S}_{n,m}(g,h)$.

	The following two theorems provide results on the starlikeness of the functions related to $\widehat{f}$ and the $n$-ply points function $F_{n}$ respectively.
	\begin{theorem}\label{thm3.1}
		If the function $h$ is a convex univalent function with positive real part satisfying $h(0)=1$, and $\widehat{f} \in  \mathcal{ST S}_{n,m}(h)$, then the function $\widehat{g}$ defined by $\widehat{g}(z)=(\widehat{f}(z)-\widehat{f}(-z) )/2$ belongs to the class $\mathcal{ST}_{n,m}(h)$.
	\end{theorem}
	\begin{proof}
		Let $\widehat{f}\in \mathcal{ST S}_{n,m}(h)$ and for $F_{n}$ given by \eqref{eqn2.2}, the function $G_{n}:\mathbb{D}\rightarrow \mathbb{C}$ be defined by $G_{n}(z)=\left.\left(F_{n}(z)-F_{n}(-z)\right)\middle/2\right.$. By definition, $zf'_{k}(z)/G_{n}(z)\in h(\mathbb{D})$ for every $z\in \mathbb{D}$.  After replacing  $z$ by $-z$ it is seen that $zf'_{k}(-z)/G_{n}(z)\in h(\mathbb{D})$ for every $z\in \mathbb{D}$, as $G_{n}(-z)=-G_{n}(z)$. Using the convexity of the domain $h(\mathbb{D})$, it is inferred that
		\[ \dfrac{zg'_{k}(z) }{G_{n}(z)}=\dfrac{1}{2}\left(\dfrac{zf'_{k}(z)}{G_{n}(z)}+\dfrac{zf'_{k}(-z)}{G_{n}(z)}\right) \in h(\mathbb{D}),\]
		for every $z\in \mathbb{D}$. Thus, the definition of the class $\mathcal{ST}_{n,m}(h)$ gives $\widehat{g}\in \mathcal{ST}_{n,m}(h)$.
	\end{proof}
	\begin{theorem}\label{thm3.2}
		Let $h$ be a convex univalent function with positive real part satisfying $h(0)=1$ and define the function $G_{n}$ as in Theorem~\ref{thm3.1}. Then, for $\widehat{f}\in \mathcal{ST S}_{n,m}(h)$, the function $G_{n}\in \mathcal{ST}(h)$.
	\end{theorem}
	\begin{proof}
		Let $\widehat{f}\in \mathcal{ST S}_{n,m}(h)$, and $z\in \mathbb{D}$. Then by definition we have $zf'_{k} (z)/G_{n}(z) \in h(\mathbb{D})$. Since $\alpha_k$'s are non-negative real numbers with $\sum_{k=0}^{m} \alpha_k=1$, the convexity property of $h(\mathbb{D})$ gives $\sum_{k=1}^{m} \alpha_k zf'_{k} (z)/G_{n} (z) \in h(\mathbb{D})$ for each $z\in \mathbb{D}$, and the definition in \eqref{eqn2.1} infers $zF' (z)/G_{n}(z) \in h(\mathbb{D})$ for each $z\in \mathbb{D}$.
		
		Replacing $z$ by $\varepsilon ^{j} z$, and using the relation $F_{n}(\varepsilon ^{j} z)=\varepsilon ^{j} F_{n}(z)$ which in  it can be seen that
		\begin{equation}\label{eqn3.1}
			\dfrac{zF'(\varepsilon ^{j} z)}{G_{n}(z)} \in h(\mathbb{D})
		\end{equation}
		for each $z\in \mathbb{D}$.
		From the definition in \eqref{eqn2.2}, it also follows that $F'_{n}(z):=\sum_{j=0}^{n-1}(F'(\varepsilon^{j}z))/n$. Applying it together with the convexity of $h(\mathbb{D})$ on \eqref{eqn3.1} yields
		\[\dfrac{zF'_{n}(z)}{G_{n}(z)}=\dfrac{1}{n}\sum_{j=0}^{n-1}\dfrac{zF'(\varepsilon ^{j} z)}{G_{n}(z)} \in h(\mathbb{D}) \]
		for each $z\in \mathbb{D}$.
		The convex combination of the functions obtained after replacing $z$ by $-z$  gives \[\dfrac{zG'_{n}(z)}{G_{n}(z)}=
		\dfrac{1}{2}\left(\dfrac{zF'_{n}(z)}{G_{n}(z)}+\dfrac{zF'_{n}(-z)}{G_{n}(z)}\right)\in h(\mathbb{D})\]
		for each $z\in \mathbb{D}$, as $G_{n}(-z)=-G_{n}(z)$. Therefore $zG'_{n}(z)/G_{n}(z) \prec h(z)$ or $G_{n}\in \mathcal{ST}(h)$.
	\end{proof}
	
	The following theorem proves a containment relationship between the classes $\mathcal{CV S}_{n,m}(h) $ and $\mathcal{ST S}_{n,m}(h)$.
	\begin{theorem}
		If   $h$ is  a convex univalent function with positive real part satisfying $h(0)=1$, then   $\mathcal{CV S}_{n,m}(h) \subset \mathcal{ST S}_{n,m}(h)$.
	\end{theorem}
	\begin{proof}
		Let the function $G_{n}:\mathbb{D}\rightarrow \mathbb{C}$ be defined as in Theorem~\ref{thm3.1}. For $\widehat{f} \in \mathcal{CV S}_{n,m} (h)$, it is known that $\left.\left((2zf'_{k})'(z)\right)\middle/\left(F'_{n}(z)+F'_{n}(-z)\right)\right.\prec h(z)$ for every $z\in \mathbb{D}$, or $(z f'_{k})'(z)/G_{n}'(z) \prec h(z)$ for every $z\in \mathbb{D}$. By arguments similar to the ones in Theorem~\ref{thm3.2}, it can be proved that $G_{n}\in \mathcal{CV}(h)$. Since $h$ is a function with positive real part, it follows that $G_{n}$ is convex, and hence starlike. By Theorem~\ref{thm1.2}, $z f'_{k}(z)/G_{n}(z) \prec h(z)$ for every $z\in \mathbb{D}$. Thus, by definition, $\widehat{f} \in \mathcal{ST S}_{n,m}(h)$.
	\end{proof}	
	In the next theorem, it is shown that the classes $\mathcal{ST S}_{n,m}(g,h)$ and $\mathcal{CV S}_{n,m}(g,h)$ are closed under convolution with the class of pre-starlike functions of order $\alpha$, where $h$ is a convex univalent function with real part greater than $\alpha$.
	\begin{theorem}
		Let $g$ be a fixed function in $\mathcal{A}$. Let $h$ be a convex univalent function satisfying $\RE h(z)>\alpha, h(0)=1,\ 0\leqslant \alpha < 1$. Then the classes $\mathcal{ST S}_{n,m}(g,h)$ and $\mathcal{CV S}_{n,m}(g,h)$ are closed under convolution with pre-starlike functions of order $\alpha$.
	\end{theorem}
	\begin{proof}
		The class $\mathcal{ST S}_{n,m}(h)$ is closed under convolution with pre-starlike functions of order $\alpha$ is proved first; that is, it will be shown that $f\ast \phi \in \mathcal{ST S}_{n,m}(h)$ for every $f\in \mathcal{ST S}_{n,m}(h)$ and every $\phi \in \mathcal{R}_{\alpha}$. For $k=1,2\ldots ,m$ and  $n\geqslant 1$, the functions $G_{n}$ and $H_{k}$ are defined on the unit disc as
		\[G_{n}(z):=\dfrac{F_{n}(z)-F_{n}(-z)}{2}, \quad \quad H_{k}(z):=\dfrac{zf'_{k}(z)}{G_{n}(z)}.\]
		It can be seen from Theorem~\ref{thm3.2} that $G_{n}\in \mathcal{ST}(h)$ and since $\RE h(z)>\alpha,$\ \ $G_{n}\in \mathcal{ST} (\alpha)$.\\
		It is observed that
		\begin{equation}\label{eqn3.2}
			\dfrac{(\phi \ast (H_{k} G_{n}))(z)}{(\phi \ast G_{n})(z)} = \dfrac{2z(\phi \ast f_{k})'(z)}{(\phi \ast F)_{n}(z)-(\phi \ast F)_{n}(-z)}
		\end{equation}
		for $z\in \mathbb{D}$. Theorem~\ref{thm1.1} yields
		\[\dfrac{(\phi \ast (H_{k} G_{n}))(z)}{(\phi \ast G_{n})(z)} \subset \overline{co}(H_{k}(\mathbb{D})),\]
		and $H_{k}(z)\prec h(z)$ . This implies that
		\[\dfrac{2z(\phi \ast f_{k})'(z)}{(\phi \ast F)_{n}(z)-(\phi \ast F_{n})(-z)} \prec h(z)\] or $ \phi \ast \widehat{f} \in \mathcal{ST S}_{n,m} (h)$.
		Thus, the class $\mathcal{ST S}_{n,m}(h)$ is   closed under convolution with functions in $\mathcal{R}_{\alpha}$.
		
		Let $\widehat{f} \in \mathcal{ST S}_{n,m} (g,h)$. Then by definition, $\widehat{f} \ast g\in \mathcal{ST S}_{n,m}(h)$. Using the associative property of convolution and the closure of the class $\mathcal{ST S}_{n,m} (h)$ under convolution with functions in $\mathcal{R}_{\alpha}$, it follows that $ (\widehat{f} \ast \phi  )\ast g= (\widehat{f} \ast g )\ast \phi \in \mathcal{ST S}_{n,m} (h) $ where $\phi\in \mathcal{R}_\alpha$. Therefore, $\widehat{f} \ast \phi \in \mathcal{ST S}_{n,m}(g,h)$, and thus proves that $\mathcal{ST S}_{n,m}(g,h)$ is closed under convolution with functions in the class $\mathcal{R}_{\alpha}$.
		
		Let $\widehat{f} \in \mathcal{CV S}_{n,m} (g,h)$. Then, by the Alexander relationship between the  classes $\mathcal{ST S}_{n,m} (g,h)$ and $\mathcal{CV S}_{n,m} (g,h)$, we have   $  z\widehat{f'} \in \mathcal{ST S}_{n,m}(g,h)$. By using what is just proved for $\mathcal{ST S}_{n,m} (g,h)$, it yields that $(z \widehat{f'} \ast \phi) \in \mathcal{ST S}_{n,m} (g,h)$, and since $(z \widehat{f'} \ast \phi)=z(\phi \ast \widehat{f})'$, it follows that $z(\phi \ast \widehat{f})'\in \mathcal{ST S}_{n,m} (g,h)$. Again, the Alexander relationship between the classes $\mathcal{ST S}_{n,m} (g,h)$ and $\mathcal{CV S}_{n,m} (g,h)$ implies that $\phi \ast \widehat{f} \in \mathcal{CV S}_{n,m}(g,h)$. This proves that the class $\mathcal{CV S}_{n,m}(g,h)$ is closed under convolution with functions in $\mathcal{R}_{\alpha}$. Since $\mathcal{CV S}_{n,m}(z/(1-z),h)=\mathcal{CV S}_{n,m}(h)$, the closure under convolution with functions in the class of pre-starlike functions of order $\alpha$ for the class $\mathcal{CV S}_{n,m}(h)$ follows directly.\end{proof}
	
	In the forthcoming sections, we define and discuss the corresponding classes of functions by considering conjugate and symmetric conjugate points, respectively.
	
	\section{The Classes $\mathcal{ST C}_{n,m} (h)$ and $\mathcal{CV C}_{n,m}(h)$}
	Let $m\ \text{and}\ n$ be positive integers, and $h$ be a convex univalent function with positive real part  which is symmetric with respect to the real-axis. The classes $\mathcal{ST C}_{n,m} (h)$  and  $\mathcal{CV C}_{n,m}(h)$ of starlike and convex functions with respect to conjugate points related to $\widehat{f}$ and $F_{n}$ consist of all functions $\widehat{f}\in \mathcal{A}^{m}$ satisfying the subordinations
	\[\dfrac{2zf'_{k} (z)}{F_{n}(z)+\overline{F_{n}(\overline{z})}}\prec h(z)\quad \text{ and }\quad  \dfrac{(2zf'_{k})'(z)}{F'_{n}(z)+\overline{F_{n}' (\overline{z})}} \prec h(z)\]
	respectively, where the function $F_{n}$ is defined in \eqref{eqn2.2}.
	For a fixed $g\in \mathcal{A}$, the class $\mathcal{ST C}_{n,m}(g,h)$ contains all functions $\widehat{f}\in \mathcal{A}^{m}$ for which $\widehat{f} \ast g\in \mathcal{ST C}_{n,m}(h)$, and\ $\mathcal{CV C}_{n,m} (g,h)$ consists of all functions $\widehat{f}\in \mathcal{A}^{m}$ for which $\widehat{f}\ast g\in \mathcal{CV C}_{n,m}(h)$.

	If $g(z)=z/(1-z)$, then $\mathcal{ST C}_{n,m}(g,h)=\mathcal{ST C}_{n,m}(h)$, $\mathcal{CV C}_{n,m} (g,h)=\mathcal{CV C}_{n,m} (h)$ and $\mathcal{ST C}_{n,m} (z/(1-z)^2,h)=\mathcal{CV C}_{n,m} (h)$. Furthermore, the following Alexander relationship holds: $\widehat{f}\in \mathcal{CV C}_{n,m} (h)$ if and only if $z\widehat{f'}\in \mathcal{ST C}_{n,m}(h)$ and  $\widehat{f}\in \mathcal{CV C}_{n,m} (g,h)$ if and only if  $z\widehat{f'}\in \mathcal{ST C}_{n,m}(g,h)$.
	
	The next two theorems study the starlikeness property of the functions relating to $\widehat{f}\in \mathcal{A}^{m}$ and the $n$-ply points function $F_{n}$.
	\begin{theorem}\label{thm4.1}
		If $h$ is a convex univalent function with positive real part  which is symmetric with respect to the real axis satisfying $h(0)=1$, and the function $\widehat{g}$ is defined by $\widehat{g}(z)=\left. \left(\widehat{f}(z)+\overline{\widehat{f}(\overline{z})}\right)\middle / 2\right.$, then $\widehat{f} \in  \mathcal{ST C}_{n,m}(h)\ \text{implies}\  \widehat{g}\in \mathcal{ST}_{n,m}(h)$.
	\end{theorem}
	\begin{proof}
		Let $\widehat{f}\in \mathcal{ST C}_{n,m}(h)$. Define $G_{n}:\mathbb{D}\rightarrow \mathbb{C}$ by $G_{n}(z)=(F_{n}(z)+\overline{F_{n}(\overline{z}))})/2$, where $F_{n}$ is given by \eqref{eqn2.2}. Then by definition, $zf'_{k}(z)/G_{n}(z)\in h(\mathbb{D})$ for every $z\in \mathbb{D}$. Replacing $z$ by $\overline{z}$, taking the conjugate of the resulting expression and then using the facts that $G_{n}(z)=\overline{G_{n}(\overline{z})}$ and $h$ is symmetric with respect to the real axis, the expression becomes $z\overline{f'_{k} (\overline{z})}/G_{n}(z)\in h(\mathbb{D})$ for every $z\in \mathbb{D}$. The convexity property of  $h(\mathbb{D})$ gives
		\[\dfrac{zg'_{k}(z) }{G_{n}(z)}=\dfrac{1}{2}\left(\dfrac{zf'_{k}(z)}{G_{n}(z)}+\dfrac{z\overline{f'_{k} (\overline{z})}}{G_{n}(z)}\right) \in h(\mathbb{D})\]
		for every $z\in \mathbb{D}$.
		Thus, by the definition of the class $\mathcal{ST}_{n,m}(h)$, it follows that $\widehat{g}\in \mathcal{ST}_{n,m}(h)$.
	\end{proof}
	
	\begin{theorem}\label{thm4.2}
		Let $h$ be a convex univalent function with positive real part  which is symmetric with respect to the real-axis satisfying $h(0)=1$ and the function $G_{n}$ be defined as in Theorem~\ref{thm4.1}. If $\widehat{f}\in \mathcal{ST C}_{n,m}(h)$, then $G_{n}\in \mathcal{ST}(h)$.
	\end{theorem}
	\begin{proof}
		Let $f\in \mathcal{ST C}_{n,m}(h)$. Then,\ $zf'_{k} (z)/G_{n}(z)\in h(\mathbb{D})$ for every $z\in \mathbb{D}$. Since $\alpha_k$'s are non-negative real numbers with $\sum_{k=0}^{m} \alpha_k=1$, with the convexity property of $h(\mathbb{D})$ it is seen that $\sum_{k=1}^{m} \alpha_k zf'_{k} (z)/G_{n} (z)\in h(\mathbb{D})$ for every $z\in \mathbb{D}$. Using \eqref{eqn2.1} it is simplified to $zF' (z)/G_{n}(z) \in h(\mathbb{D})$ for each $z\in \mathbb{D}$.
		
		By \eqref{eqn2.2}, it is already observed in the proof of Theorem~\ref{thm2.2} that $F_{n}(\varepsilon ^{j} z)=\varepsilon ^{j} F_{n}(z)$; using this identity after replacing $z$ by $\varepsilon ^{j} z$, the expression becomes
		\begin{equation}\label{eqn4.1}
			\dfrac{zF'(\varepsilon ^{j} z)}{G_{n}(z)} \in h(\mathbb{D}).
		\end{equation}
		for each $z\in \mathbb{D}$. By using $F'_{n}(z):=\sum_{j=0}^{n-1}(F'(\varepsilon^{j}z))/n$, which is again a consequence of \eqref{eqn2.2}, along with the convexity of $h(\mathbb{D})$ in \eqref{eqn4.1}, it gives
		\[\dfrac{zF'_{n}(z)}{G_{n}(z)}=\dfrac{1}{n}\sum_{j=0}^{n-1}\dfrac{zF'(\varepsilon ^{j} z)}{G_{n}(z)} \in h(\mathbb{D})\]
		for every $z\in \mathbb{D}$.
		Replacing $z$ by $\overline{z}$, taking conjugate of the resulting expresssion and using the facts that $G_{n}(z)=\overline{G_{n}(\overline{z})}$ and that $h$ is symmetric with respect to the real axis yields
		\begin{equation}\label{eqn4.2}
			\dfrac{z\overline{F'_{n} (\overline{z})}}{G_{n}(z)}\in h(\mathbb{D})
		\end{equation}
		for each $z\in \mathbb{D}$. By taking the convex combination of \eqref{eqn4.1} and \eqref{eqn4.2}, we see that
		\[\dfrac{zG'_{n}(z)}{G_{n}(z)}=\dfrac{1}{2}\dfrac{z(F'_{n}(z)+\overline{F'_{n} (\overline{z})})}{G_{n}(z)}\in h(\mathbb{D})\] or in other words, $zG'_{n}(z)/G_{n}(z) \prec h(z)$  for every $z\in \mathbb{D}$. Thus, it follows that $G_{n}\in \mathcal{ST}(h)$.
	\end{proof}
	
	The following theorem gives the containment result for the classes $\mathcal{CV C}_{n,m}(h)$ and $\mathcal{ST C}_{n,m}(h)$.
	\begin{theorem}
		If   $h$ is  a convex univalent function with positive real part  which is symmetric with respect to the real-axis satisfying $h(0)=1$, then   $\mathcal{CV C}_{n,m}(h) \subset \mathcal{ST C}_{n,m}(h)$.
	\end{theorem}
	\begin{proof}
		Let $\widehat{f} \in \mathcal{CV C}_{n,m} (h)$ and define the function $G_{n}:\mathbb{D} \rightarrow \mathbb{C}$ as in Theorem~\ref{thm4.1}. Then $(z f'_{k})'(z)/G_{n}'(z) \prec h(z)$ for each $z\in \mathbb{D}$. By a process similar to the one in Theorem~\ref{thm4.2}, it can be concluded that $G_{n}\in \mathcal{CV}(h)$. Since $h$ is given to be a function with positive real part, it follows that $G_{n}$ is convex and hence starlike. An application of Theorem~\ref{thm1.2} gives  $z f'_{k}(z)/G_{n}(z) \prec h(z)$ for each $z\in \mathbb{D}$, that is $\widehat{f} \in \mathcal{ST C}_{n,m}(h)$.
	\end{proof}	
	In the following theorem, a result related to closure under convolution with functions in the class $\mathcal{R}_\alpha$, for the functions in $\mathcal{ST C}_{n,m}(g,h)$ and other related classes is discussed.
	
	\begin{theorem}
		Let $g$ be a fixed function in $\mathcal{A}$ and let $h$ be a convex univalent function which is symmetric with respect to the real-axis satisfying $\RE h(z)>\alpha, h(0)=1,\ 0\leqslant \alpha < 1$. Then the classes $\mathcal{ST C}_{n,m}(g,h)$ and $\mathcal{CV C}_{n,m}(g,h)$ are closed under convolution with the functions in the class of pre-starlike functions of order $\alpha$ having real coefficients.
	\end{theorem}
	\begin{proof}
		First, the statement of the theorem is proved to be true for the class $\mathcal{ST C}_{n,m}(h)$, that is, it is shown that the functions in $\mathcal{ST C}_{n,m}(h)$ are closed under convolution with pre-starlike functions of order $\alpha$ with real coefficients. Let $\widehat{f}\in \mathcal{ST C}_{n,m}(h)$ and $\phi \in \mathcal{R}_{\alpha}$ has real coefficients. For $k=1,2\ldots ,m$ define the functions $G_{n}$ and $H_{k}$ on $\mathbb{D}$ as
		\[G_{n}(z):=\dfrac{F_{n}(z)+\overline{F_{n}(\overline{z})}}{2}, \quad \quad H_{k}(z):=\dfrac{zf'_{k}(z)}{C_{n}(z)}.\]
		By Theorem~\ref{thm4.2}, $G_{n}\in \mathcal{ST}(h)$, and since $\RE h(z)>\alpha$,\ \ $G_{n}\in \mathcal{ST}(\alpha)$.  Theorem~\ref{thm1.1} gives
		\begin{equation}\label{eqn4.3}
			\dfrac{(\phi \ast (H_{k} G_{n}))(z)}{(\phi \ast G_{n})(z)} \subset \overline{co}(H_{k}(\mathbb{D})).
		\end{equation}
		The fact that $\phi$ has real coefficients provides
		\begin{equation}\label{eqn4.4}
			\dfrac{(\phi \ast (H_{k} G_{n}))(z)}{(\phi \ast G_{n})(z)}=  \dfrac{2z(\phi \ast f_{k})'(z)}{(\phi \ast F)_{n}(z)+\overline{(\phi \ast F)_{n}(\overline{z})}}
		\end{equation}
		for every $z\in \mathbb{D}$. Thus, from \eqref{eqn4.3} and \eqref{eqn4.4} together with the fact that $H_{k}(z)\prec h(z)$, it can be inferred that
		\[\dfrac{2z(\phi \ast f_{k})'(z)}{(\phi \ast F)_{n}(z)+\overline{(\phi \ast F)_{n}(\overline{z})}} \prec h(z)\]
		which implies that $\widehat{f}\ast \phi \in \mathcal{ST C}_{n,m}(h)$. It proves that the class $\mathcal{ST C}_{n,m}(h)$  is closed under convolution with functions in the class $\mathcal{R}_{\alpha}$ having real coefficients.
		
		For $\widehat{f} \in \mathcal{ST C}_{n,m} (g,h)$, we have  $\widehat{f} \ast g\in \mathcal{ST C}_{n,m}(h)$. The associative property of convolution gives $ (\widehat{f} \ast \phi  )\ast g= (\widehat{f} \ast g )\ast \phi$ where $\phi\in \mathcal{R}_\alpha$ with real coefficients, and the fact that the class $\mathcal{ST C}_{n,m} (h)$ is closed under convolution with pre-starlike functions of order $\alpha$ with real coefficients infers that $(\widehat{f} \ast \phi  )\ast g \in \mathcal{ST C}_{n,m} (h)$. Therefore,  $\widehat{f} \ast \phi \in \mathcal{ST C}_{n,m}(g,h)$.
		
		For $\widehat{f} \in \mathcal{CV C}_{n,m} (g,h)$, it is shown that $\phi \ast \widehat{f} \in \mathcal{CV C}_{n,m}(g,h)$. Using the Alexander relationship between the  classes $\mathcal{ST C}_{n,m} (g,h)$ and $\mathcal{CV C}_{n,m} (g,h)$, it is inferred that $  z\widehat{f'} \in \mathcal{ST C}_{n,m}(g,h)$. Using the result proved above for $\mathcal{ST C}_{n,m} (g,h)$, it follows that $(z \widehat{f'} \ast \phi) \in \mathcal{ST C}_{n,m} (g,h)$ where $\phi \in \mathcal{R}_{\alpha}$ with real coefficients. As $(z \widehat{f'} \ast \phi)=z(\phi \ast \widehat{f})'$,\ so  $z(\phi \ast \widehat{f})'\in \mathcal{ST C}_{n,m} (g,h)$. Again, the Alexander relationship between the  classes $\mathcal{ST C}_{n,m} (g,h)$ and $\mathcal{CV C}_{n,m} (g,h)$  gives $\phi \ast \widehat{f} \in \mathcal{CV C}_{n,m}(g,h)$, proving that the class $\mathcal{CV C}_{n,m}(g,h)$ is closed under convolution with pre-starlike functions of order $\alpha$ with real coefficients. The result that the class $\mathcal{CV C}_{n,m}(h)$ is closed under convolution with functions in $\mathcal{R}_{\alpha}$ having real coefficients follows since $\mathcal{CV C}_{n,m}(z/(1-z),h)=\mathcal{CV C}_{n,m}(h)$.
	\end{proof}

	\section{The Classes $\mathcal{ST SC}_{n,m} (h)$ and $\mathcal{CV SC}_{n,m} (h)$}
	For positive integers $m$ and $n$, consider the $n$-ply function $F_n$ defined in \eqref{eqn2.2} and the function $h$, a convex univalent function with positive real part  which is symmetric with respect to the real-axis. Using all these functions, the class $\mathcal{ST SC}_{n,m} (h)$ of starlike functions with respect to symmetric conjugate points related to $\widehat{f}$ and $F_{n}$ is defined by
	\[\mathcal{ST C}_{n,m} (h)=\left\{\widehat{f} \in \mathcal{A}^{m}\ |\ \dfrac{2zf'_{k} (z)}{F_{n}(z)-\overline{F_{n}(\overline{-z})}}\prec h(z)\right\}.\]
	The class $\mathcal{ST SC}_{n,m}(g,h)$ has all those $\widehat{f}$ for which $\widehat{f} \ast g\in \mathcal{ST SC}_{n,m}(h)$. The class $\mathcal{CV SC}_{n,m} (h)$ of convex function with respect to symmetric conjugate points related to $\widehat{f}$ and $F_{n}$ consists of all $\widehat{f}\in \mathcal{A}^{m}$ satisfying the subordination
	\[\dfrac{(2zf'_{k})'(z)}{F'_{n}(z)+\overline{F_{n}' (\overline{-z})}} \prec h(z).\]
	The class $\mathcal{CV SC}_{n,m} (g,h)$ consists of all $\widehat{f}\in \mathcal{A}^{m}$ for which $\widehat{f}\ast g\in \mathcal{CV SC}_{n,m} (h)$.
	
	It can be easily inferred that $\mathcal{ST SC}_{n,m}(z/(1-z),h)=\mathcal{ST SC}_{n,m}(h)$, $\mathcal{CV SC}_{n,m} (z/(1-z),h)=\mathcal{CV SC}_{n,m} (h)$ and $\mathcal{ST SC}_{n,m} (z/(1-z)^2,h)=\mathcal{CV SC}_{n,m} (h)$. Also, the following version of Alexander's theorem holds: $\widehat{f}\in \mathcal{CV SC}_{n,m} (h)$ if and only if $z\widehat{f'}\in \mathcal{ST SC}_{n,m}(h)$ and  $\widehat{f}\in \mathcal{CV SC}_{n,m} (g,h)$ if and only if  $z\widehat{f'}\in \mathcal{ST SC}_{n,m}(g,h)$.
	
	By combining the techniques used to prove the results in Section $3$ and Section $4$, the corresponding results for the class $\mathcal{ST SC}_{n,m} (h)$ and other related classes can be  proved and their proofs are omitted.
	
	\begin{theorem}
		Let the function $h$ is a convex univalent function with positive real part which is symmetric with respect to the real-axis satisfying $h(0)=1$, also define the function $\widehat{g}$ by $\widehat{g}(z)= (\widehat{f}(z)-\overline{\widehat{f}(\overline{-z})})/2$. If $\widehat{f} \in  \mathcal{ST SC}_{n,m}(h),\ \text{then}\  \widehat{g}\in \mathcal{ST}_{n,m}(h)$.
	\end{theorem}
	
	\begin{theorem}
		Let $h$ be a convex univalent function with positive real part  which is symmetric with respect to the real-axis satisfying $h(0)=1$, and the function $G_{n}:\mathbb{D}\rightarrow \mathbb{C}$ be defined by $G_{n}(z)= (F_{n}(z)-\overline{F_n(\overline{-z})}) /2 $. Then, for $\widehat{f}\in \mathcal{ST SC}_{n,m}(h)$, the function $ G_{n}\in \mathcal{ST}(h)$.
	\end{theorem}
	
	\begin{theorem}
		If   $h$ is a convex univalent function with positive real part which is symmetric with respect to the real-axis satisfying $h(0)=1$, then   $\mathcal{CV SC}_{n,m}(h) \subset \mathcal{ST SC}_{n,m}(h)$.
	\end{theorem}
	
	\begin{theorem}
		Let $g$ be a fixed function in $\mathcal{A}$ and let $h$ be a convex univalent function which is symmetric with respect to the real-axis satisfying $\RE(h(z))>\alpha, h(0)=1,\ 0\leqslant \alpha < 1$. Then the classes $\mathcal{ST SC}_{n,m}(g,h)$ and $\mathcal{CV SC}_{n,m}(g,h)$ are closed under convolution with pre-starlike functions of order $\alpha$ having real coefficients.
	\end{theorem}
	
	\begin{remark}
		For $m=1$, the four classes $\mathcal{ST}_{n,m} (g,h)$, $\mathcal{ST S}_{n,m} (g,h)$, $\mathcal{ST C}_{n,m}(g,h)$ and $\mathcal{ST SC}_{n,m}(g,h)$ reduce to the classes discussed in \cite{ravi1}. For $m=n=1,\text{ and } g(z)=z/(1-z)$ these classes reduce to those in \cite{ravi2}. For $m=1$, these become the classes studied in \cite{RosAbeerRavi} with $p=1$.
		
	\end{remark}

\noindent\textbf{Acknowledgment.} The authors are thankful to the referee for the comments to improve the readability of the paper.

\end{document}